\documentclass[11pt, oneside, 
a4paper]{article}
\usepackage[ansinew]{inputenc}
\usepackage[english]{babel}
\usepackage{amsbsy,amscd,amsfonts,amsmath,amsopn,amssymb,amstext,amsthm,amsxtra,color,fancybox,float,graphicx,latexsym,subfigure,srcltx,times,tikz,url}
\newcommand{\mo}{\R^{n\times n}_{zd}}
\newcommand{\hip}{\{x_n=0\}}

\newcommand{\sucen}[2]{{#1_{1},\hdots,#1_{#2}}}

\newcommand{\realamp}{\R\cup\{-\infty\}}

\newcommand{\colut}[3]{\left[\begin{array}{r}#1\\#2\\#3
\end{array}\right]}
\newcommand{\colud}[2]{\left[\begin{array}{r}#1\\#2
\end{array}\right]}

\newcommand{\m}{\medskip}

\newcommand{\N}{\mathbb{N}}

\newcommand{\R}{\mathbb{R}}

\newcommand{\PP}{\mathcal{P}}

\renewcommand{\int}{\operatorname{int}}

\newcommand{\dd}{\operatorname{d}}
\newcommand{\dH}{\operatorname{d_H}}

\newcommand{\rr}{\operatorname{r}}

\newcommand{\col}{\operatorname{col}}

\newcommand{\card}{\operatorname{card}}

\newcommand{\spann}{\operatorname{span}}

\renewcommand{\labelenumi}{\theenumi}

\renewcommand{\labelenumi}{(\roman{enumi})}

\newtheorem{thm}{Theorem}
\newtheorem{lem}[thm]{Lemma}

\newtheorem{ex}[thm]{Example}

\title{On tropical Kleene star matrices and alcoved polytopes
}
\author{
M.J. de la Puente
}

\date{} 

\begin{document}
\maketitle


\begin{abstract}
In this paper we give a short, elementary proof of a known  result in tropical mathematics, by which the convexity of the column span of a  zero--diagonal real matrix $A$ is characterized by  $A$ being a Kleene star. We give  applications to alcoved polytopes, using normal idempotent matrices (which form a subclass of Kleene stars). For a normal matrix we define a norm  and show that this is the radius of a hyperplane section of its tropical span.
\end{abstract}
\section{Introduction}
Tropical algebra (also called max--algebra, extremal algebra, etc.) is a linear algebra performed with  the so called tropical operations:  $\max$ (for addition) and $+$ (for multiplication)---though some variations  use $\min$ instead of $\max$, or ordinary multiplication as tropical multiplication. The study of tropical algebra   began in the 60's and 70's with the works of Cuninghame--Green, Gondran--Minoux,  Vorobyov, Yoeli and K. Zimmermann and has received a fabulous push since the 90's. Today it  ramifies into other areas such as algebraic geometry and mathematical analysis. Tropical algebra began as a means to mathematically model processes which involve synchronization of machines. Applications to such practical problems are still pursued today.

\m
A basic problem in tropical algebra is to determine the properties (classical or tropical) of the
set  $V$  spanned (by means of tropical operations) by  $m$ given points  $\sucen am$ in $\R^n$. The properties of $V$ follow
from the properties of the $n\times m$ real matrix $A$ given by the
coordinates of the $a_j$ written in columns. In this setting, $V$ is denoted $\spann(A)$. It is always a connected, compact set, and most often it is non--convex, in the classical sense.
Convexity--related questions about $\spann(A)$ have drawn  the attention of  various authors; see \cite{Develin_Sturm,Jimenez_P,Joswig_K,Sergeev}, as well as \cite{Izhakian_al_Idempotent,Izhakian_al_Pure}.

\m
Assume $m=n$. Kleene operators (also called Kleene stars or Kleene closures) are well--known in mathematical logic and computer science. For matrices in tropical algebra, Kleene stars (meaning matrices which are Kleene stars of other matrices) form a particularly well--behaved class.
They are simply characterized in terms of linear equalities and inequalities.  
For a given  matrix  $A$, it is customary for authors to obtain properties of $A$ (and $\spann(A)$) from properties of the directed graph $G_A$ associated to $A$; see
\cite{Akian_HB,Baccelli,Butkovic_Libro,Cuninghame,Cuninghame_New, Zimmermann_K}. For example,  the tropical  (or max--algebraic) principal eigenvalue $\lambda(A)$ of $A$ is the maximum cycle mean of $G_A$. But if $A$ is a Kleene star,  then properties of $\spann(A)$ follow directly from $A$: we need not consider $G_A$.

\m
Alcoved polytopes form a very natural class of generally non--regular convex polytopes, including hypercubes and cross polytopes. They have been studied in  \cite{Lam_Postnikov,Lam_PostnikovII,Werner_Yu}. An alcoved polytope  directly arises from a Kleene star matrix.

\m
In this note we prove, by elementary handling of inequalities, the following known result:
 for any zero--diagonal real matrix $A$, $A$ is a Kleene star if and only if $\spann(A)$ is convex. Since a certain hyperplane section of $\spann(A)$ is an alcoved polytope,  we are able to obtain some applications to these. One application is the possibility of using tropical operations in order  to compute the numerous extremals (vertices and pseudovertices) of a given alcoved polytope. Another application is a way to improve the presentation of an alcoved polytope. A third application is the computation of the radius of an alcoved polytope.

\section{Kleene stars, column spans and normal idempotent matrices}

\m
Write $\oplus=\max$ and $\odot=+$. These are the tropical operations addition and multiplication.
For $n\in\N$, set $[n]:=\{1,2,\ldots,n\}$.
Let  $\R^{n\times m}$ denote the set of  real matrices having $n$ rows and $m$ columns. Define tropical sum and product of matrices following the same rules of classical linear algebra, but replacing addition (multiplication) by tropical addition (multiplication).
We will never use classical sum or multiplication of matrices, in this note; therefore, $A\odot B, A\odot A$ will be written $AB, A^2$, respectively,  for matrices $A,B$. Besides, we will never use the classical linear span.


\m
We will write the coordinates of points in $\R^n$ in columns.
Let $A\in \R^{n\times m}$ and denote by $a_1,\ldots,a_m\in \R^n$ the columns of $A$.
The \emph{tropical column span} of $A$ is, by definition,
\begin{eqnarray}
\spann(A):&=\{(\lambda_1+a_1)\oplus\cdots\oplus (\lambda_m+a_m) \in \R^n: \lambda_1,\ldots,\lambda_m\in\R\}\\
\nonumber&=\max\{\lambda_1+a_1,\ldots,\lambda_m+a_m: \lambda_1,\ldots,\lambda_m\in\R\}
\end{eqnarray}
where maxima are computed coordinatewise. For instance, $$\left(3+\colud{-2}{1}\right)\oplus\left( 0 +\colud{2}{1}\right)=\colud{1}{4}\oplus\colud{2}{1}=\colud{2}{4}, \text{\ so\ that\ } \colud{2}{4}\in \spann\left[\begin{array}{rr}
-2&2\\
1&1
\end{array}\right].$$

Notice that, by definition, the set $\spann(A)$ is closed under classical addition  of the vector $(\lambda,\ldots,\lambda)$, for $\lambda\in\R$. Therefore, a hyperplane section of it, such as $\spann(A)\cap\hip$ determines $\spann(A)$.

We will mostly consider real zero--diagonal square matrices, in this paper. The set of such matrices will be denoted $\mo$.
For $A=(a_{ij})\in \mo$, consider the matrix $A_0=(\alpha_{ij})$, \label{dfn:alpha} where
\begin{equation}\label{eqn:alpha}
\alpha_{ij}=a_{ij}-a_{nj},
\end{equation}
whence $\col(A_0,j)=-a_{nj}+\col(A,j)$. The columns of $A_0$ belong to the hyperplane $\hip$ and are tropical scalar multiples of the columns of $A$, so that
\begin{equation}\label{eqn:span}
\spann(A)=\spann(A_0).
\end{equation}
Thus, $x\in\spann(A)\cap \hip$ if and only if there exist $\sucen \mu n\in\R$ such that
\begin{align}
x_j=&\max_{k\in[n]}\{\alpha_{jk}+\mu_k\},\quad j\in[n-1],\label{eqn:max_otra}\\
0=&\max_{k\in[n]}\mu_k,\label{eqn:max_l}
\end{align}
so that $x$ is a combination of the columns of $A_0$ with coefficients $\mu_j$ (tropically) adding up to zero.

\m
  By definition (see \cite{Butkovic_Sch_Ser,Sergeev,Sergeev_Sch_But}), $A\in\mo$ is \emph{a Kleene star} if  $A=A^2$ (i.e., $A$ is zero--diagonal and idempotent, tropically). If each diagonal entry of $A=(a_{ij})$ vanishes, then  $A\le A^2$, because   for each $i,j\in[n]$, we have
$$a_{ij}\le \max_{k\in[n]} a_{ik}+a_{kj}=(A^2)_{ij}.$$
Therefore, being a Kleene star is characterized by the following  $n$ linear equalities and ${n\choose 2}+{n\choose 3}=\frac{n^3-n}{6}$
linear inequalities:
\begin{equation}\label{eqn:Kleene}
a_{ii}=0,\quad a_{ik}+a_{kj}\le a_{ij}, \quad i,j,k\in[n], \quad \card\{i,j,k\}\ge2.
\end{equation}
In particular, $a_{ik}+a_{ki}\le0$, for $i,k\in[n]$.

\m

By definition, an \emph{alcoved polytope} $\PP $ in $\R^{n-1}$ is a convex polytope defined by inequalities $a_{i}\le x_i\le b_{i}$ and $a_{ik}\le x_i-x_k\le b_{ik}$ , for some $i,k\in[n-1]$, $i\neq k$, and $a_i, b_i,a_{ik}, b_{ik}\in\R\cup\{\pm\infty\}$. 
The polytope $\PP $ may have up to ${2n-2}\choose {n-1}$ extremals (in the sense of classical convexity) and this bound is sharp; see \cite{Develin_Sturm}. This is a fast--growing number, since
$${{2n}\choose {n}}\simeq\frac{4^n}{\sqrt{\pi n}},$$ as ${n\to\infty}$, by Stirling's formula. For instance, for $n=10$, $\PP$ may have up to 48.620 extremals.

A matrix $A\in\mo$  induces the following  (possibly empty!) alcoved polytope  in $\R^{n-1}$
\begin{equation}\label{eqn:C_A}
C_A:=\left\{x\in\R^{n-1}:\  {{a_{in}\le x_i\le-a_{ni}}\atop {a_{ik}\le x_i-x_k\le-a_{ki}}};\  i,k\in[n-1], i\neq k\right\}.
\end{equation}

\m
Throughout the paper, we \emph{identify} $\R^{n-1}$ with the hyperplane $\{x_n=0\}$ in $\R^n$. Our main result is

\begin{thm}\label{thm:main}
For any $A\in \mo$, the following are equivalent:
\begin{enumerate}
\item $A$ is a Kleene star,\label{item:1}
\item  $C_A=\spann(A)\cap \hip$.\label{item:2}
\end{enumerate}
\end{thm}

To prove this theorem we need two lemmas. Given two points $x,y\in \R^n$, let $B\in\R^{n\times 2}$ be the matrix whose columns are $x$ and $y$. The set $\spann(B)$ is called the \emph{tropical segment} joining $x$ and $y$ (not to be confused with the tropical line determined by $x$ and $y$).


\begin{lem}\label{lem:2}
If $A\in \mo$, then  $C_A\subseteq \spann(A)\cap\hip$.
\end{lem}
\begin{proof}
Given  $x=(x_1,\ldots,x_{n-1})^t\in C_A$, write $x_n=0$ and consider scalars $\mu_n=0$ and $\mu _i=x_i+a_{ni}\le0$, for $i\in[n-1]$. Then (\ref{eqn:max_otra}) and (\ref{eqn:max_l})  hold true,  due to (\ref{eqn:alpha}) and to the
$n(n-1)$ inequalities defining $C_A$. Thus, $x\in\spann(A)\cap\hip$.
\end{proof}

\begin{lem}[Tropical convexity of $C_A$] \label{lem:3} If $A\in \mo$, then $\spann(B)\cap \hip\subseteq C_A$,
for every $x,y$ in $C_A$.
\end{lem}
\begin{proof} Assume that $x,y\in C_A$.
 A point $z$ in $\spann(B)\cap\hip$
has coordinates $z_n=0=\max\{\lambda,\mu\}$ and
\begin{equation*}
z_i=\max\{\lambda+x_i,\mu+y_i\},\quad i\in[n-1],
\end{equation*}
for  some $\lambda,\mu\in\R$.

Say $\lambda=0, \mu\le0$; then
$$x_i\le \max\{x_i,\mu+y_i\}=z_i\le \max\{x_i,y_i\},\quad i\in[n-1],$$ so that
$$a_{in}\le z_i\le -a_{ni},\quad i\in[n-1].$$

Moreover, if $i,k\in[n-1]$, $i\neq k$,
we have
$$z_i-z_k=\left\{\begin{array}{ll}
x_i-x_k,&\text{\ if\ }x_i=z_i,\ x_k=z_k,\\
y_i-y_k,&\text{\ if\ }\mu+y_i=z_i,\ \mu+y_k=z_k,\\
\end{array}\right.$$
and
$$x_i-x_k\le z_i-z_k=\mu+ y_i-x_k\le y_i-y_k,$$
if $\mu+y_i=z_i$, $x_k=z_k$.
In any case, we get
$$a_{ik}\le  z_i-z_k\le-a_{ki}.$$
\end{proof}

\m
Now we go to the proof of theorem \ref{thm:main},  showing that (i) and (ii) are also equivalent to
\begin{enumerate}\setcounter{enumi}{2}
\item each column of $A_0$ belongs to $C_A$.\label{item:3}
\end{enumerate}

\begin{proof}
Recall that $A_0=(\alpha_{ij})$, where $\alpha_{ij}=a_{ij}-a_{nj}$. Then, for $i,j\in[n]$,

\renewcommand{\labelenumi}{(\alph{enumi})}
\begin{enumerate}
\item $\alpha_{ni}=0$,  $\alpha_{in}=a_{in}$ and $\alpha_{ii}=-a_{ni}$,
\item $\alpha_{ij}-\alpha_{jj}=a_{ij}$.
\end{enumerate}

If $A$ is a Kleene star, then  $a_{ii}=0$ and $a_{ik}+a_{kj}\le a_{ij}$, so that
\begin{enumerate} \setcounter{enumi}{2}
\item $a_{in}\le\alpha_{ij}\le-a_{ni}$, \label{it:cota1}
\item $a_{ik}\le\alpha_{ij}-\alpha_{kj}=a_{ij}-a_{kj}\le-a_{ki}$. \label{it:cota2}
\end{enumerate}
Items  (c)  and (d)  mean precisely that  each column of $A_0$ belongs to $C_A$, so  we have that
\renewcommand{\labelenumi}{(\roman{enumi})}
 \ref{item:1} is equivalent to \ref{item:3}.

The coordinates $(\sucen x{n-1},0)^t$ of a point $x$ in $\spann(A)\cap\hip$  satisfy  $x_j=\max_{k\in[n]}\{\alpha_{jk}+\mu_k\}$, with $0=\max_{k\in[n]}\mu_k$. Say, without loss of generality,  $\mu_1=0$ and write
\begin{equation*}
x=z\oplus (\mu _3+\col(A_0,3))\oplus\cdots\oplus (\mu _n+\col(A_0,n)),
\end{equation*} with
$z=\col(A_0,1)\oplus (\mu _2+\col(A_0,2))$.
Assuming \ref{item:3}, then $z$  lies in $C_A$, by lemma \ref{lem:3}.  Again by lemma \ref{lem:3}, in finitely many steps,
we show that $x$ lies in $C_A$. Thus,  \ref{item:3} implies \ref{item:2}, by lemma \ref{lem:2}. And \ref{item:2} implies \ref{item:3}, because $\spann(A)=\spann(A_0)$.


\end{proof}

Theorem \ref{thm:main} and its proof deal with linear inequalities and maxima, because the equivalence between conditions \ref{item:1} and \ref{item:2}  can be restated as
$$(\ref{eqn:Kleene}) \Leftrightarrow [x\in C_A \Leftrightarrow \exists \mu_1,\ldots,\mu_n \text{\ such \ that\ \ }(\ref{eqn:max_otra}) \text{ \ and\ } (\ref{eqn:max_l})]$$ and $x\in C_A$ (see (\ref{eqn:C_A})) depends on inequalities.

\m

The convex set $C_A\subseteq \R^{n-1}=\{x_n=0\}$   gives rise to another  convex subset in $\R^n$ as follows: $\overline{C_A}=\{(x,0)+(\lambda,\ldots,\lambda):x\in C_A, \lambda\in \R\}$, the Minkowski sum of $C_A$ and a line. It is obvious that
\begin{enumerate}\setcounter{enumi}{3}
\item $\overline{C_A}=\spann(A)$
\end{enumerate}
is equivalent to \ref{item:2} in theorem \ref{thm:main}.

\m
Theorem \ref{thm:main} (and its equivalent item (iv)) is closely related to Sergeev's  section 3.1 in \cite{Sergeev} (please note that the notation in \cite{Sergeev} is multiplicative ---i.e., $\odot$ is the usual multiplication). In particular,  see top of p. 324 and propositions 3.4, 3.5 and 3.6. In terms of that work, we are proving that a zero--diagonal matrix $A$ is a Kleene star if and only if  its column span equals its subeigenvector cone (denoted $V^*(A)$ in \cite{Sergeev} and $\overline{C_A}$ here). In proposition 3.4 in \cite{Sergeev},  the assumption is that $A$ is definite, meaning  that $\lambda(A)=0$. In proposition 3.5, the assumption is that $A$ is strongly definite, meaning that $\lambda(A)=0$ and $a_{ii}=0$, $i\in[n]$. There,  $\lambda(A)$ denotes the maximum cycle mean of $A$,  the cycles referring to the directed graph  $G_A$.  And $\lambda(A)$ happens to be the unique eigenvalue of $A$.  Sergeev's result and proof  can also be found in p.26 of \cite{Butkovic_Libro}.
Unlike in \cite{Butkovic_Libro,Sergeev}, we are not using the terminology of max--plus spectral theory or multi--order convexity to present or explain our main result (although this is possible too). Moreover, we are not assuming anything about $\lambda(A)$.

\m
Theorem \ref{thm:main} is  also related to  proposition 3.6 in \cite{Werner_Yu}, where a different concept of generating set for an alcoved polytope is considered (please note that in \cite{Werner_Yu}, $\oplus$ means minimum).

\m

A \emph{first application} to alcoved polytopes $\PP \subset \R^{n-1}$  goes as follows. Remember that  $\PP $ is a convex set (in the classical sense) having a large number  $s$ of extremals:  $s\le {{2n-2}\choose {n-1}}$. If $\PP =C_A$ for some Kleene star $A\in\mo$,  we know that $\PP $  is tropically spanned by  the $n$ columns of $A_0$. The columns of $A_0$ are extremals of $\PP$ of course, the advantage being that  the remaining $s-n$  extremals of $\PP$  can be computed from $A_0$, using a tropical algorithm, such as  \cite{Allamigeon_al}. Some authors call \emph{vertices} to the columns of $A_0$ and  \emph{pseudovertices} to the  remaining $s-n$  extremals of $\PP$.

\begin{ex}\label{ex:1}
The alcoved polytope $\PP \subset \R^2$ (see figure \ref{fig:01}, left) given by
$$-1\le x\le 3, \quad -2\le y\le 6,\quad -4\le y-x\le 5$$
satisfies $\PP =C_A$, with
$$A=\left[\begin{array}{rrr}
0&-5&-1\\-4&0&-2\\-3&-6&0
\end{array}\right], \quad A_0=\left[\begin{array}{rrr}
3&1&-1\\-1&6&-2\\0&0&0
\end{array}\right].$$
Since $A=A^2$,   then $\PP $ is  spanned by the columns of $A_0$. In particular, the three columns of $A_0$ are extremals of $\PP $. The other  three extremals of $\PP $ are combinations of these. To be precise,
$$\colut{3}{6}{0}=\colut{3}{-1}{0}\oplus\colut{1}{6}{0},\colut{-1}{4}{0}=-2+\colut{1}{6}{0}\oplus\colut{-1}{-2}{0},\colut{2}{-2}{0}=
\colut{-1}{-2}{0}\oplus-1+\colut{3}{-1}{0}.$$
\end{ex}

\begin{figure}[H]
 \centering
  \includegraphics[width=10cm,keepaspectratio]{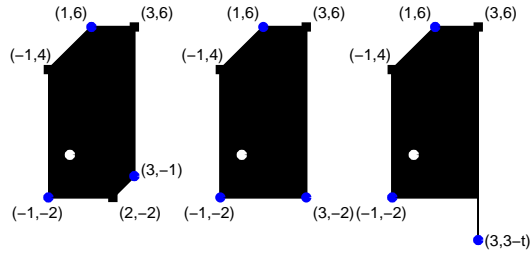}\\
  \caption{Alcoved polytopes in examples \ref{ex:1}, \ref{ex:3} and \ref{ex:5}. Generators are rounded (in blue), other extremals are squared (in black), the origin is marked (in white).}
  \label{fig:01}
  \end{figure}

 \begin{ex}\label{ex:2}
 Let $\PP =C_A\subset \R^3$ (see figure \ref{fig:02}), where $$A=\left[\begin{array}{rrrr}
    0&-6&-10&-5\\
    -8&0&-5&-3\\
    -3&-5&0&-6\\
    -5&-3&-6&0\\
    \end{array}\right],\quad A_0=\left[\begin{array}{rrrr}
    5&-3&-4&-5\\
    -3&3&1&-3\\
    2&-2&6&-6\\
    0&0&0&0\\
    \end{array}\right].$$
 Since $A=A^2$, then the columns of $A_0$ span $\PP $, i.e, they are extremals of $\PP $ and every other extremal of $\PP $ can be computed tropically from them (as tropical combinations).
    It can be checked (with the help of a computer program) that $C_A$ has $17<{6\choose3}=20$ extremals: the coordinates of the remaining 13 extremals  are the columns of the matrix
$$\left[\begin{array}{rrrrrrrrrrrrrrrrr}
   -5&    -3&     5&     5&     1&     5&    -3&    -3&    -4& -5&    -5&    -5&    -5\\
        1&    -1&     3&     3&     3&     1&    -3&     3&     2&   1&    -1&     0&    -3\\
        5&    -6&     6&     2&    -2&     6&    -6&     6&     6& -4&    -6&     5&     2\\
         0&     0&     0&     0&     0&     0&     0&     0&     0&   0&     0&     0&     0
\end{array}\right]$$
\end{ex}
\begin{figure}[H]
 \centering
  \includegraphics[width=11cm,keepaspectratio]{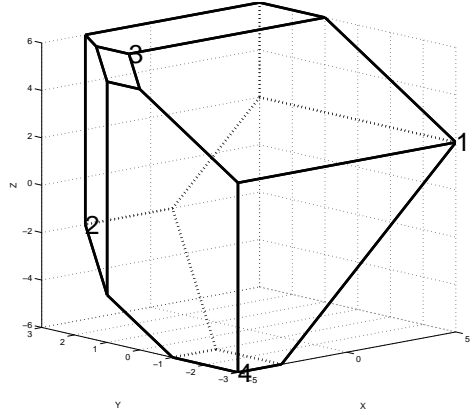}\\
  \caption{Alcoved polytope from example \ref{ex:2}. The columns of $A_0$ are marked  with digits $1,2,3$ and 4.}
  \label{fig:02}
  \end{figure}

 Theorem \ref{thm:main} deals with Kleene stars, but we prefer to work with  a subclass of particularly nice matrices. These are the \emph{normal idempotent matrices} (NI, for short). By definition, a real matrix $A=(a_{ij})$ is \emph{normal} if  $a_{ii}=0$, $a_{ij}\le 0$, all $i,j\in [n]$; see \cite{Butkovic_Libro}.  Notice that if $A$ is NI, then  $a_{ik}+a_{kj}\le a_{ij}$, for all $i,j,k\in[n]$, so that  $A$ is a Kleene star, by (\ref{eqn:Kleene}). The converse is not true; for instance, $A=\left[\begin{array}{rr}
0&-2\\1&0\\
\end{array}\right]$ is a Kleene star but not a normal matrix. A NI matrix $A$ satisfies $\lambda(A)=0$, although we do not need this. 

\m
Clearly,   $A$ is normal if and only if   $C_A$ contains  the origin, in which case, by lemma \ref{lem:2}, $\spann(A)$ does too.
Informally speaking, a matrix $A$ is normal if the columns  of $A_0$ are set around the origin of $\R^{n-1}$, and they follow a precise order ---and this order is a kind of orientation in $\R^{n-1}$.

\m
Due to the \emph{Hungarian method}\label{dfn:hungarian} (see \cite{Kuhn,Papa}), any order $n$ real matrix $A$ can be normalized, meaning that there exist (non necessarily unique) order $n$ matrices $P,Q,N$ such that $N=QAP$ and $N$ is normal. Moreover, $\spann(N)$ has the same properties of $\spann(A)$, since multiplication by $P$  amounts to a relabeling of columns, and multiplication by $Q$  amounts to performing a translation. (Here are a  few words on the properties of $P$ and $Q$. The matrices $P$ and $Q$ are \emph{generalized permutation matrices}. Here we extend $\R$ to $\realamp$.  A \emph{diagonal} matrix is $D=(d_{ij})$ with $d_{ii}\in\R$ and $d_{ij}=-\infty$.  In particular, if $d_{ii}=0$ for all $i\in[n]$, we get a matrix which acts as an \emph{identity} for  matrix multiplication, since $-\infty$ acts as a neutral element for $\oplus=\max$. A \emph{generalized permutation matrix} is the result of applying a permutation $\sigma$ to the rows and columns of a diagonal matrix).  We need \emph{not} use matrices over $\realamp $ in this paper, because when normalizing  a given $A\in \mo$, every instance of $-\infty$ in the matrices $P$ and $Q$ above can be replaced by $-t$, for some real number $t\ge0$ big enough, yielding real matrices $P'$ and $Q'$ with $N=Q'AP'$; see remark 2 in p. \pageref{dfn:hungarian2} for a bound on $t$.)   The matrix $N$ can be obtained from $A$ with $O(n^3)$ elementary tropical operations (max and +); see \cite{Butkovic_Libro} and therein.

\m
A pioneer paper dealing with normal matrices is \cite{Yoeli} (although another terminology is used there).
If $A$ is normal, then clearly $A\le A^2\le A^3\le\ldots$ and Yoeli proved in \cite{Yoeli} that $A^{n-1}=A^{n}=A^{n+1}=\cdots$, so that $A^{n-1}$ is NI, so is a Kleene star. Denote this matrix by $A^*$ and call it \emph{the Kleene star of $A$}. More generally, for any real square matrix $A$, define $A^*$ as   $A\oplus A^2\oplus A^3\oplus\cdots$, if this limit  exists in $\R^{n\times n}$.

\begin{lem}\label{lem:mejora}
If $A^*$ exists, then $C_A=C_{A^*}$.
\end{lem}
\begin{proof}
By the Hungarian method, we may suppose that $A$ is normal, so that $A^*=A^{n-1}$. Clearly, $C_A\supseteq C_{A^{n-1}}$, because $A\le A^{n-1}$. To prove the converse, assume that $A<A^2$. Then there exist pairwise different $i,j,k\in[n]$ such that $a_{ik}<a_{ij}+a_{jk}=\max_s a_{is}+a_{sk}$. Suppose that $x\in C_A$; then
\begin{align}
a_{ij}&\le x_i-x_j\le-a_{ji},\label{eqn:one}\\
a_{kj}&\le x_k-x_j\le-a_{jk},\label{eqn:twos}\\
a_{ik}&\le x_i-x_k\le-a_{ki}.\label{eqn:three}
\end{align}
Subtracting (\ref{eqn:twos}) from (\ref{eqn:one}), we get
$$(A^2)_{ik}=a_{ij}+a_{jk}\le x_i-x_k$$
which improves (\ref{eqn:three}) to
$$(A^2)_{ik}\le x_i-x_k\le-a_{ki}.\label{eqn:four}$$
 By going through every entry for which $A$ and $A^2$ differ and improving the inequalities as we just did, we get $C_A=C_{A^2}$. In a finite number of steps, we get the desired result.
\end{proof}

Lemma \ref{lem:mejora} provides a  \emph{second application} to alcoved polytopes $\PP $. A given presentation $C_A$ of   $\PP$ can be improved to a  tight presentation $\PP =C_{A^*}$.

\begin{ex}\label{ex:3}
The alcoved polytope $\PP \subset \R^2$ (see figure \ref{fig:01}, center) determined by
$$-1\le x\le 3, \quad -2\le y\le 6,\quad y-x\le 5$$
gives rise to the  matrix
$$A=\left[\begin{array}{rrr}
0&-5&-1\\-\infty&0&-2\\-3&-6&0
\end{array}\right]$$ or, in order to have a real  matrix, we can write
$$A(t)=\left[\begin{array}{rrr}
0&-5&-1\\-t&0&-2\\-3&-6&0
\end{array}\right],$$ for $t\in\R$ big enough.
Now,
$$A(t)^2=\left[\begin{array}{rrr}
0&-5&-1\\-5&0&-2\\-3&-6&0
\end{array}\right]$$ is idempotent and does not depend on $t$. Write $A(t)^2=A(t)^*=A^*$. Then, by lemma \ref{lem:mejora}, $\PP =C_{A^*}$ and $A^*$ describes $\PP $ tightly. Moreover, by theorem \ref{thm:main},
$\PP $ is  spanned by the columns of
$$(A^*)_0=\left[\begin{array}{rrr}
3&1&-1\\-2&6&-2\\0&0&0
\end{array}\right].$$
\end{ex}

Notice that in the proof of proposition 3.6 of \cite{Werner_Yu}, the authors assume that an alcoved polytope $C_A$ is described by tight inequalities and then    they show that $A$ is a Kleene star (without explicitly mentioning it).

\bigskip

We close this note by pointing out some
some nice \emph{features of normal and NI matrices.}

If $A$ is NI, then the columns of $(-A^T)_0$ are extremals of $\spann(A)\cap \{x_n=0\}$.   A proof of this fact is found in  \cite{Jimenez_P} for $n=4$, but the proof works in general.
This can be checked out in our examples \ref{ex:1} and \ref{ex:3} (see also the corresponding figures):
$$(-A^T)_0=\left[\begin{array}{rrr}
-1&2&3\\4&-2&6\\0&0&0
\end{array}\right], \quad(-(A(t)^2)^T)_0=\left[\begin{array}{rrr}
-1&3&3\\4&-2&6\\0&0&0
\end{array}\right]$$
and in example \ref{ex:2}, where the first four columns of the $4\times 13$ matrix are precisely the columns of $(-A^T)_0$.

\begin{figure}[H]
 \centering
  \includegraphics[width=6cm,keepaspectratio]{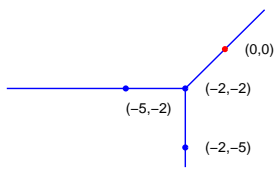}\\
  \caption{Tropical line in $\R^2$ with vertex at the point $(-2,-2)$.}
  \label{fig:03}
  \end{figure}

For $p\in\R^{n}$,  set
$$||p||:=\max_{i,j\in[n]} \{|p_i|, |p_i-p_j|\}.$$ This is a \emph{seminorm} in $\R^n$ (meaning that the property  $||\lambda+p||=|\lambda|+||p||$, for $\lambda\in\R$ is not required). The seminorm $||\cdot||$  is \emph{invariant} under the embedding of $\R^{n-1}\simeq\{x_n=0\}\subset \R^n$. It gives rise to a  \emph{semidistance} in $\R^n$ (where the property $\dd (p,q)=0\Rightarrow p=q$ is not required)
\begin{equation}\label{dfn:trop_dist}
\dd (p,q):=\max_{i,j\in[n]}\{|p_i-q_i|, |p_i-q_i-p_j+q_j|\}.
\end{equation}
This is a \emph{distance} on the hyperplane $\R^{n-1}\simeq \{x_n=0\}$! It measures the integer length (or lattice length) of the tropical segment $\spann(p,q)$.
In $\R^2\simeq \{x_3=0\}$, for example, we have $\dd ((-2,-2),(0,0))=2$ (not $2\sqrt{2}$!), $\dd ((-5,-2),(-2,-5))=\max\{3,6\}=6=3+3$  and $\dd (-5,-2),(0,0))=\max\{5,2,3\}=5=3+2$.
It is a sort of Manhattan distance; see figure \ref{fig:03}. 

\m
Define the \emph{tropical radius} of  a subset $S\subset\R^{n-1}$ containing the origin,  as follows:
\begin{equation}
\rr(S):=\sup_{s\in S}\dd (s,0)=\sup_{s\in S} ||s||.
\end{equation}

For a matrix $A$, consider
\begin{equation}\label{eqn:norm}
|||A|||:=\max_{i,j}|a_{ij}|.
\end{equation}
If $A$ is normal, then  $a_{ii}=0$ and $a_{ij}\le0$,    so that  $A\le A^{n-1}$, whence $|||A|||\ge ||| A^{n-1}|||$.

\m
Below we prove that the radius of $C_A$ equals the norm of $A$, for a NI matrix $A$.

\begin{thm}\label{thm:radius}
If $A$ is normal, then $|||A|||=\rr(\spann(A)\cap\hip)$. If, in addition,  $A$ is idempotent, then $|||A|||=\rr(C_A)$.
\end{thm}
\begin{proof}
We only need prove the first statement.

We know that  $A_0=(\alpha_{ij})$, with $\alpha_{ij}=a_{ij}-a_{nj}$.
Assume that $A=(a_{ij})$ is normal (i.e., $a_{ii}=0$ and $a_{ij}\le0$). We first prove that
\begin{equation}\label{eqn:igual}
|||A|||=\max_{k\in[n]}||\col(A_0,k)||.
\end{equation}
To do so, write $M$ for the maximum on the right hand side.
We have
\begin{equation}\label{eqn:maxima}
M=\max_{i,j,k\in[n]}\{|\alpha_{ik}|, |\alpha_{ik}-\alpha_{jk}|\}=\max_{i,j,k\in[n]} |a_{ik}-a_{jk}|.
\end{equation}
Using $a_{ii}=0$, we get $|||A|||\le M$. On the other hand, the maximum on the right hand side of (\ref{eqn:maxima}) cannot be achieved for mutually different $i,j,k$ since $a_{ik}\le0$ and $a_{jk}\le0$; thus we get $|||A|||=M$.


From equalities (\ref{eqn:span}) and  (\ref{eqn:igual}), we obtain $|||A|||\le\rr(\spann(A)\cap\hip)$.

Now, assume that $p,y$ are two columns of $A_0$ and let $z=\lambda+p\oplus \mu +y$, with $z_n=0=\max\{\lambda,\mu\}$. Say $\lambda=0$. Then
$$z_j=\max\{p_j,\mu+y_j\}\le\max\{p_j,y_j\}\le\max\{|p_j|,|y_j|\}\le \max\{||p||,||y||\}.$$ Besides, by the same argument used in the proof of lemma \ref{lem:3}, we get $p_i-p_k\le z_i-z_k\le y_i-y_k$,
proving that $||z||\le \max\{||p||,||y||\}\le M=|||A|||$.
\end{proof}

\m
Remark 1: It is easy to check that (\ref{eqn:norm}) defines a matrix norm on $\mo$ endowed with $\oplus$, $\odot$, but we do not use it here.

Remark 2: \label{dfn:hungarian2} In the Hungarian method mentioned in p. \pageref{dfn:hungarian}, it is customary to write matrices $P,Q$ with entries in $\realamp$, while $A,N$ are real. However, every instance of $-\infty$ in $P,Q$ can be replaced by $-t\in\R$, with $t>\!\!>|||A|||, |||N|||$, getting  $P',Q'$  real such that $N=Q'AP'$.

Remark 3: In \cite{Cuninghame_B,Sergeev_def}, the range seminorm $\tau$ in $\R^n$ is introduced as follows: $\tau(p)=\max_{i,j\in[n]} p_i-p_j=\max_{i,j\in[n]} |p_i-p_j|$. In general, $\tau(p)\le ||p||$.
The seminorm $\tau$  is \emph{not} invariant under the embedding of $\R^{n-1}\simeq\{x_n=0\}\subset \R^n$.
The range seminorm  gives rise to a semidistance, used in \cite{Cohen,Sergeev_def}, and denoted $\dH $. The distances induced by $\dd $ and $\dH $ on $\hip$ coincide. It is  a tropical version of  Hilbert's projective distance.

\begin{ex}\label{ex:4}
Let  $$B=\left[\begin{array}{rrrr}
    0&-6&-10&-5\\
    -9&0&-5&-3\\
    -3&-5&0&-6\\
    -5&-3&-6&0\\
    \end{array}\right],$$
    then $B^2=A$ of example \ref{ex:2} and $\spann(B)$ is not convex. We have $|||B|||=|||A|||=10$  so that the sets $\spann(B)\cap\{x_4=0\}$ and $C_A$ have both  radius 10.
\end{ex}

\begin{ex}\label{ex:5}
Returning to example \ref{ex:3}, the radius of  $\spann A(t)\cap\{x_3=0\}$ is $t$, for $t\ge6$, while  the radius of $C_{A(t)}=C_{A^*}$ is $6$. This is clear from  figure \ref{fig:01} right, where the non--convex set $\spann A(t)$ has an arbitrary long \lq\lq antenna".
\end{ex}

\m

Remark 4: In section 4 of \cite{Sergeev_def}, Sergeev computes the radius of a $\dH$--ball inscribed in $\spann(A)$. Sergeev computes the biggest ball fitting inside  $\spann (A)$ and we compute a ball centered at the origin and containing $\spann (A)$; see figure \ref{fig:04}.

\begin{figure}[H]
 \centering
  \includegraphics[width=10cm,keepaspectratio]{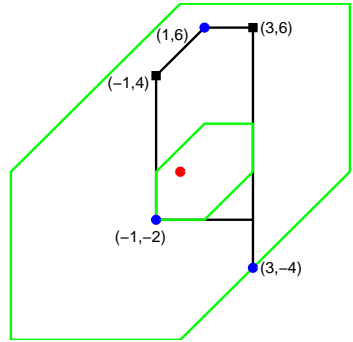}\\
  \caption{$\spann A(7)$ from example \ref{ex:3} (in black) and  balls of radius 2 and 7 fitting inside and outside (in green).}
  \label{fig:04}
  \end{figure}

\section*{Acknowledgements}
I am indebted to S. Sergeev for very helpful discussions and to two referees for taking an interest on this paper and pointing out some ways of improving the manuscript.

\centerline{\footnotesize{M. J. de la Puente. Dpto. de Algebra. Facultad de Matem\'{a}ticas. Universidad Complutense. Madrid. Spain. \texttt{mpuente@mat.ucm.es}}}

\end{document}